\documentclass[reqno]{amsart}
\usepackage[psamsfonts]{amssymb}
\usepackage{amsthm}
\usepackage{color}

\newtheorem{theorem}{Theorem}
\newtheorem{lemma}{Lemma}

\theoremstyle{remark}

\theoremstyle{definition}

 \theoremstyle{definition}
 \newtheorem{defn}{Definition}
 \theoremstyle{remark}

 \numberwithin{equation}{section}

\usepackage{graphicx}%
\usepackage{multirow}%
\usepackage{amsmath,amssymb,amsfonts}%
\usepackage{amsthm}%
\usepackage{mathrsfs}%
\usepackage[title]{appendix}%
\usepackage{xcolor}%
\usepackage{textcomp}%
\usepackage{manyfoot}%
\usepackage{algorithm}%
\usepackage{algorithmicx}%
\usepackage{algpseudocode}%
\usepackage{listings}%

 \usepackage{graphicx,amssymb,upref}
\usepackage{graphicx}
\usepackage{float}
\usepackage{enumerate}
\usepackage{hyperref}

\def\bege{\begin{equation}} \def\ende{\end{equation}}

   \def\begr{\begin{eqnarray}}
\def\endr{\end{eqnarray}} 
 
\def\bege{\begin{equation}} \def\ende{\end{equation}}
\def\begr{\begin{eqnarray}} \def\endr{\end{eqnarray}}
\def\bnum{\begin{enumerate}} \def\enum{\end{enumerate}}

\allowdisplaybreaks[4]

\begin{document}
\title[Complex symmetric operator ]{Complex symmetric generalized product type operators on weighted Bergman space}
\thanks{The first author is thankful to  DST(India) for WISE-PDF  (File no. DST/WISE-PDF/PM-46/2023.}
\keywords {Weighted composition-differentiation operators; Berg  spaces; Normal operator; Self-adjoint operator; Complex symmetric operator}
\subjclass[2010]{47B33, 47B38, 46E10, 32A37.}

\author[A. Sharma]{Aakriti
 Sharma}

\address{
Department of   Mathematics, University of Jammu-180006, J\&{K},
 India.}

\email{aakritishma321@gmail.com}

\begin{abstract}
   In this paper, inspired by \cite{HS1}, we provide a characterization of complex symmetric, self-adjoint, and normal generalized product-type operators on weighted Bergman space, with respect to conjugation $JW_{u_ {\alpha}, \vartheta_ {\alpha}}$.
 \end{abstract}

\maketitle

\section{Introduction}\label{sec:intr}
Let $\mathcal{H}(\mathbb{D})$ represent the set of all analytic functions defined on the unit disk $\mathbb{D} = \{ z \in \mathbb{C} : |z| \leq 1 \}$, and let $\mathfrak{S}(\mathbb{D})$ denote the set of all holomorphic self-maps of $\mathbb{D}$. For any parameter $\beta$ satisfying $-1 < \beta \le \infty$, the weighted Bergman space, denoted by $\mathcal{A}^{2}_{\beta}$, consists of all functions $f$ in $\mathcal{H}(\mathbb{D})$ such that
\begin{equation*}
  \|f\|^{2}_{ \mathcal{A}^{2}_{\beta}}=\int_{\mathbb{D}}|f(z)|^{2}d\mathcal{A}_{\beta}(z)<\infty  
\end{equation*}

where $d\mathcal{A}_{\beta}(z)$ is an associated weighted area measure on the unit disk.

%Throughout the paper, let $\mathcal{H}(\mathbb{D})$ be the space of all analytic functions on the unit disk $\mathbb{D}=\{z\in \mathbb{C}:|z|\le 1\}$ and $\mathfrak{S}(\mathbb{D})$ be the space of all self maps in $\mathbb{D}$. For $-1 < \alpha \le \infty$, the Bergman space is denoted by $\mathcal{A}^{2}_{\alpha}$ and is defined as 
%\begin{equation}
  %  \mathcal{A}^{2}_{\alpha}=\Bigg\{f: f\in \mathcal{H}(\mathbb{D}) \text{ and } \|f\|^{2}_{ \mathcal{A}^{2}_{\alpha}}=\int_{\mathbb{D}}|f(z)|^{2}d\mathcal{A}_{\alpha}(z)<\infty\Bigg\}
%\end{equation}
%Here $d\mathcal{A}_{\alpha}(z)=(\alpha +1) (1-|z|^{2})^{\alpha}d\mathcal{A}(z)$ is the weighted area measure with standard weights $(\alpha +1) (1-|z|^{2})^{\alpha}, \;\; \alpha > -1 $ on $\mathbb{D}$.

The weighted Bergman space $\mathcal{A}^{2}_{\beta}$, for suitable values of $\beta$, is a well-known example of a reproducing kernel Hilbert space (RKHS). The inner product on this space is defined by
\[
\langle f, g \rangle = \int_{\mathbb{D}} f(z)\, \overline{g(z)}\, d\mathcal{A}_{\beta}(z),
\]
for all functions $f, g \in \mathcal{A}^{2}_{\beta}$. The reproducing kernel at a point $\xi \in \mathbb{D}$ is given by
\[
K_{\xi}(z) = \frac{1}{(1 - \overline{\xi}z)^{\beta + 2}}.
\]
This reproducing property ensures that for any $f \in \mathcal{A}^{2}_{\beta}$ and $\xi \in \mathbb{D}$,
\[
\langle f, K_{\xi} \rangle = f(\xi), \quad \text{and} \quad \langle f, K_{\xi}^{(n)} \rangle = f^{(n)}(\xi),
\]
where $K_{\xi}^{(n)}$ denotes the reproducing kernel associated with the $n^{\text{th}}$ derivative at the point $\xi$, and is explicitly defined by
\[
K_{\xi}^{(n)}(z) = \frac{z^n \prod_{i=2}^{n+1} (\beta + i)}{(1 - \overline{\xi}z)^{\beta + n + 2}}.
\]

For more details on Bergman spaces and their properties, see \cite{CM, HK}.

Let $\vartheta \in \mathfrak{S}(\mathbb{D})$ be an analytic self-map of the unit disk and $u \in \mathcal{H}(\mathbb{D})$ be an analytic function. The \emph{weighted composition differentiation operator} of order $n$ is defined by
\[
W^{n}_{u,\vartheta}(f)= W_{u,\vartheta}\mathfrak{D}^{n}(f)= u \cdot \left(f^{(n)} \circ \vartheta\right),
\]
where $\mathfrak{D}^{n}(f) = f^{(n)}$ denotes the $n^{\text{th}}$ derivative of $f$, and $W_{u,\vartheta}(f) = u \cdot (f \circ \vartheta)$ is the standard weighted composition operator. It should be noted that the differentiation operator $\mathfrak{D}$ is unbounded.

We now recall the concept of complex symmetric operators, first introduced by Garcia and Putinar \cite{GP1, GP2}.

\begin{defn}
Let $\mathcal{H}$ be a complex Hilbert space. A map $\Upsilon : \mathcal{H} \to \mathcal{H}$ is called a \emph{conjugation} if it satisfies the following properties:
\begin{enumerate}
    \item \textbf{Conjugate-linearity:} For all $\xi_1, \xi_2 \in \mathcal{H}$ and $\lambda_{1}, \lambda_{2} \in \mathbb{C}$,
    $
    \Upsilon(\lambda_{1} \xi_1 + \lambda_{2} \xi_2) = \overline{\lambda_{1}} \Upsilon(\xi_1) + \overline{\lambda_{2}} \Upsilon(\xi_2).
    $
    \item \textbf{Involution:} $\Upsilon^2 = I$, where $I$ is the identity operator.
    \item \textbf{Isometry:} $\|\Upsilon \xi\| = \|\xi\|$ for all $\xi \in \mathcal{H}$.
\end{enumerate}
\end{defn}

\begin{defn}
Let $L$ be a bounded linear operator on $\mathcal{H}$, and let $L^*$ denote its adjoint. Then:
\begin{enumerate}
    \item $L$ is called \emph{$\Upsilon$-complex symmetric} if $L = \Upsilon L^* \Upsilon$.
    \item $L$ is \emph{self-adjoint} if $L = L^*$.
    \item $L$ is \emph{normal} if $LL^* = L^*L$, or equivalently, $\|L\| = \|L^*\|$.
\end{enumerate}
\end{defn}

In recent research, the study of product type operators on spaces of analytic functions has attracted growing interest. For further developments, refer to \cite{ AK, FH1, FH2, LK} and references therein.

In this paper, we explore a new conjugation via weighted composition operators introduced in~\cite{JH}.

\textbf{Underlying Idea:} Given two conjugations $\Upsilon_1$ and $\Upsilon_2$, their composition $\Upsilon_1 \Upsilon_2$ results in a unitary operator. This observation implies that any two conjugations are related by a unitary transformation. Specifically, one can write $\Upsilon_2 = \Upsilon_1 (\Upsilon_1 \Upsilon_2)$. Motivated by this, we construct a new conjugation on the weighted Bergman space $\mathcal{A}^{2}_{\beta}$ by composing the standard conjugation $J$ with a unitary weighted composition operator $W_{u_\alpha, \vartheta_\alpha}$. The resulting operator $J W_{u_\alpha, \vartheta_\alpha}$ defines a conjugation, as shown in Lemma 2 of~\cite{JH}.

This new conjugation is given by:
$
(J W_{u_\alpha, \vartheta_\alpha})(f)(z) = \overline{\kappa \, \frac{\alpha - \overline{z}}{1 - \overline{\alpha} z} \, f\left( \frac{(1 - |\alpha|^2)^{\frac{\beta + 2}{2}}}{(1 - \overline{\alpha} z)^{\beta + 2}} \right)},
$
where the standard conjugation $J$ acts on analytic functions by
$
J(f(z)) = \overline{f(\overline{z})}.
$

The operator $W_{u_\alpha, \vartheta_\alpha}$ is a unitary weighted composition operator on $\mathcal{A}^{2}_{\beta}$, as discussed in~\cite{ZO}, and is defined by:
$
W_{u_\alpha, \vartheta_\alpha}(f)(z) = \kappa \, \frac {\alpha - z}{1 - \overline{\alpha} z} \, f\left( \frac{(1 - |\alpha|^2)^{\frac{\beta + 2}{2}}}{(1 - \overline{\alpha} z)^{\beta + 2}} \right),
$
where the weight and symbol are given respectively by:
$
 u_\alpha(z) = \kappa \, \frac {\alpha - z}{1 - \overline{\alpha} z}, \quad  \vartheta_\alpha(z) = \frac{(1 - |\alpha|^2)^{\frac{\beta + 2}{2}}}{(1 - \overline{\alpha} z)^{\beta + 2}},
$
with $\alpha \in\mathbb{D}$ and $|\kappa| = 1$.

\section{Complex symmetric and self-adjoint generalized product type operators}\label{sec:prelimi}
In this section, we characterize complex symmetric weighted composition-differentiation operators on weighted Bergman spaces with respect to the conjugation $JW_{u_b,\varphi_b}$. We then examine the self-adjoint weighted composition-differentiation operators on weighted Bergman spaces

%\begin{lemma}\cite{AH} Let  $W^{n}_{u,\vartheta}$ is bounded on $\mathcal{A}^{2}_{\alpha}.$ Then 
%\begin{equation}
   % (W^{n}_{u,\vartheta})^{*}K_{\xi}(z)= \overline{u(\xi)} K^{(n)}_{\vartheta(\xi)}(z)
%\end{equation}
  %for any $\vartheta \in \mathfrak{S}(\mathbb{D})$ and $u\in \mathcal{H}(\mathbb{D})$. 
%\end{lemma}

\begin{lemma}\cite{AH}
Suppose the operator $W^{n}_{u,\vartheta}$ is bounded on the weighted Bergman space $\mathcal{A}^{2}_{\beta}$. Then, for any $\xi \in \mathbb{D}$, the adjoint operator satisfies:

$$(W^{n}_{u,\vartheta})^{*} K_{\xi}(z) = \overline{u(\xi)} \, K^{(n)}_{\vartheta(\xi)}(z),
$$
where $\vartheta \in \mathfrak{S}(\mathbb{D})$ is an analytic self-map of the unit disk, and $u \in \mathcal{H}(\mathbb{D})$ is analytic on $\mathbb{D}$.
\end{lemma}

\begin{theorem}\label{1t}
Let $\vartheta \in \mathfrak{S}(\mathbb{D})$ and $u \in \mathcal{H}(\mathbb{D})$ be such that the operator $W^{n}_{u,\vartheta}$ is bounded on $\mathcal{A}^{2}_{\beta}$. Then, $W^{n}_{u,\vartheta}$ is complex symmetric with respect to the conjugation $J W_{u_\alpha, \vartheta_\alpha}$ on $\mathcal{A}^{2}_{\beta}$ if and only if the functions $u$ and $\vartheta$ take the following forms:
\begin{equation}
u(z) = \frac{(1 - |\alpha|^2)^{\beta + n + 2} \, u^{(n)}(\overline{\alpha}) \, (z - \overline{\alpha})^n}{\left(~1 - \alpha \vartheta(\overline{\alpha}) + z(\overline{\kappa} \vartheta(\overline{\alpha}) - \alpha)~\right)^{\beta + n + 2}},
\end{equation}
and
\begin{equation}
\vartheta(z) = \vartheta(\overline{\alpha}) + \frac{(1 - |\alpha|^2)(z - \overline{\alpha}) \vartheta'(\overline{\alpha})}{1 - \vartheta(\overline{\alpha}) \alpha + z(~\overline{\kappa} \vartheta(\overline{\alpha}) - \alpha~)}.
\end{equation}
\end{theorem}

\begin{proof}
 First, suppose that $W^{n}_{u, \vartheta}$ is complex symmetric with respect to conjugation $JW_{u_ {\alpha}, \vartheta_ {\alpha}}$. Then
 \begin{equation}\label{e1}JW_{u_ {\alpha}, \vartheta_ {\alpha}}(W^{n}_{u,\vartheta})^{*}K_{\xi}(z)= W^{n}_{u,\vartheta}JW_{u_ {\alpha}, \vartheta_ {\alpha}}K_{\xi}(z)\end{equation}
we can calculate
\begin{align}\label{11}
 JW_{u_ {\alpha}, \vartheta_ {\alpha}}(W^{n}_{u,\vartheta})^{*}K_{\xi}(z)&=  JW_{u_ {\alpha}, \vartheta_ {\alpha}}\overline{u(\xi)} K^{(n)}_{\vartheta(\xi)}(z) \notag\\
 &= \overline{u_ {\alpha}(\overline{z})}u(\xi)\overline{K^{(n)}_{\vartheta(\xi)}(\overline{z})}\notag\\
 &=\dfrac{(1-|\alpha|^{2})^{\frac{\beta+2}{2}}}{(1-  \alpha  z)^{\beta +2}}u(\xi)\dfrac{\overline{\vartheta_ {\alpha}(\overline{z})}^{n}\displaystyle\prod_{i=2}^{n+1}(\beta+i)}{(1-\vartheta(\xi)\overline{\vartheta_ {\alpha}(\overline{z})})^{\beta+n+2}}\notag\\
 &= \dfrac{(1-|\alpha|^{2})^{\frac{\beta+2}{2}}u(\xi)\overline{\kappa}^{n}(\overline{\alpha}-z)^{n}\displaystyle\prod_{i=2}^{n+1}(\beta+i)}{(1-  \alpha  z-\vartheta(\xi)\overline{\kappa}\overline{\alpha}+\vartheta(\xi)\overline{\kappa}z)^{\beta+n+2}}\notag\\
 &= \dfrac{(1-|\alpha|^{2})^{\frac{\beta+2}{2}}u(\xi)(\overline{\alpha}-z)^{n}\displaystyle\prod_{i=2}^{n+1}(\beta+i)}{\left(1-\vartheta(\xi)\overline{\kappa}\overline{\alpha}-z( \alpha -\overline{\kappa}\vartheta(\xi)\right)^{\beta +n+2}} 
\end{align}
and 
\begin{align}\label{22}
 W^{n}_{u,\vartheta}JW_{u_ {\alpha}, \vartheta_ {\alpha}}K_{\xi}(z)
 &=   W^{n}_{u,\vartheta}Ju_ {\alpha}(z)K_{\xi}(\vartheta_ {\alpha})(z)\notag\\
 &= W^{n}_{u,\vartheta}\overline{u_ {\alpha}(\overline{z})K_{\xi}(\vartheta_ {\alpha})(\overline{z})}\notag\\
 &= u(z)\left(\dfrac{(1-|\alpha|^{2})^{\frac{\beta+2}{2}}}{(1-  \alpha  z)^{\beta +2}}\dfrac{1}{\left(1-\xi\overline{\vartheta_ {\alpha}(\overline{z})}\right)^{\beta+2}}\right)^{(n)}\circ\vartheta(z)\notag\\
 &= \dfrac{(1-|\alpha|^{2})^{\frac{\beta+2}{2}}u(z)(\overline{\alpha}-\xi)^{n}\displaystyle\prod_{i=2}^{n+1}(\beta+i)}{\left(1-\vartheta(z)\overline{\kappa}\overline{\alpha}-\xi( \alpha -\overline{\kappa}\vartheta(z))\right)^{\beta +n+2}}
\end{align}

Therefore, from \eqref{e1}, \eqref{11} and \eqref{22}, we have
\begin{equation}\label{e2}
\dfrac{u(\xi)(\overline{\alpha}-z)^{n}}{\left(1-\vartheta(\xi)\overline{\kappa}\overline{\alpha}-z( \alpha -\overline{\kappa}\vartheta(\xi))\right)^{\beta +n+2}} = \dfrac{u(z)(\overline{\alpha}-\xi)^{n}}{\left(1-\vartheta(z)\overline{\kappa}\overline{\alpha}-\xi( \alpha -\overline{\kappa}\vartheta(z))\right)^{\beta +n+2}}
 \end{equation}
Set $u(z)=(z-\overline{\alpha})^{j}g(z),$ where $g$ is analytic on $\mathbb{D}$ with $g(\overline{\alpha})\neq 0$ and $j$ is a positive integer.
Then \eqref{e2} becomes
\begin{equation}\label{e3}
\dfrac{g(\xi)(\xi-\overline{\alpha})^{j-n}}{\left(1-\vartheta(\xi)\overline{\kappa}\overline{\alpha}-z( \alpha -\overline{\kappa}\vartheta(\xi))\right)^{\beta +n+2}} = \dfrac{g(z)(z-\overline{\alpha})^{j-n}}{\left(1-\vartheta(z)\overline{\kappa}\overline{\alpha}-\xi( \alpha -\overline{\kappa}\vartheta(z))\right)^{\beta +n+2}}
 \end{equation}
If $j\ge n,$ let $\xi=\overline{\alpha}$ in \eqref{e3}, we get $g\equiv 0$ which is a contradiction to the assumption that $g(\overline{\alpha})\neq 0.$ Thus $j=n$ and we obtain \eqref{e3} as
\begin{equation}\label{e4}
 \dfrac{g(\xi)}{\left(1-\vartheta(\xi)\overline{\kappa}\overline{\alpha}-z( \alpha -\overline{\kappa}\vartheta(\xi))\right)^{\beta +n+2}} = \dfrac{g(z)}{\left(1-\vartheta(z)\overline{\kappa}\overline{\alpha}-\xi( \alpha -\overline{\kappa}\vartheta(z))\right)^{\beta +n+2}}
   \end{equation}
Let $\xi=\overline{\alpha}.$ Noting that $\kappa  \alpha =\overline{\alpha},$ we get
\begin{equation}\label{e5}
 \dfrac{g(\overline{\alpha})}{\left(1-\vartheta(\overline{\alpha})\alpha-z( \alpha -\overline{\kappa}\vartheta(\overline{\alpha}))\right)^{\beta +n+2}} = \dfrac{g(z)}{\left(1-\vartheta(z) \alpha -\overline{\alpha}( \alpha -\overline{\kappa}\vartheta(z))\right)^{\beta +n+2}}
   \end{equation}
 Implies
 \begin{equation*}
     g(z)= \dfrac{(1-|\alpha|^{2})^{\beta+n+2} g(\overline{\alpha})}{\left(1- \alpha\vartheta(\overline{\alpha})+z(\overline{\kappa}\vartheta(\overline{\alpha})- \alpha)\right)^{\beta +n+2}}
 \end{equation*}
 Since $ u(z)= (z-\overline{\alpha})^{n}g(z)$, we obtain that
 \begin{equation}\label{e6}
     u(z)= \dfrac{(1-|\alpha|^{2})^{\beta+n+2} u^{(n)}(\overline{\alpha})(z-\overline{\alpha})^{n}}{\left(1- \alpha\vartheta(\overline{\alpha})+z(\overline{\kappa}\vartheta(\overline{\alpha})- \alpha)\right)^{\beta +n+2}}
 \end{equation}
 where $u^{(n)}(\overline{\alpha})= g(\overline{\alpha})\neq0.$
 From \eqref{e2} and \eqref{e6}, we have
 \begin{align}\label{e7}
    &\dfrac{(1-|\alpha|^{2})^{\beta+n+2}  u^{n}(\overline{\alpha})(\xi-\overline{\alpha})^{n}(\overline{\alpha}-z)^{n}}{\left(1- \alpha\vartheta(\overline{\alpha})+\xi(\overline{\kappa}\vartheta(\overline{\alpha})- \alpha)\right)^{\beta +n+2}\left(1-\vartheta(\xi)\overline{\kappa}\overline{\alpha}-z( \alpha -\overline{\kappa}\vartheta(\xi))\right)^{\beta +n+2}}\notag\\
    &= \dfrac{(1-|\alpha|^{2})^{\beta+n+2} u^{n}(\overline{\alpha})(z-\overline{\alpha})^{n}(\overline{\alpha}-\xi)^{n}}{\left(1- \alpha\vartheta(\overline{\alpha})+z(\overline{\kappa}\vartheta(\overline{\alpha})- \alpha)\right)^{\beta +n+2}\left(1-\vartheta(z)\overline{\kappa}\overline{\alpha}-\xi( \alpha -\overline{\kappa}\vartheta(z))\right)^{\beta +n+2}}
  \end{align}
  Implies
  \begin{align}\label{e8}
    &\left(1- \alpha\vartheta(\overline{\alpha})+z(\overline{\kappa}\vartheta(\overline{\alpha})- \alpha)\right)^{\beta +n+2}\left(1-\vartheta(z)\overline{\kappa}\overline{\alpha}-\xi( \alpha -\overline{\kappa}\vartheta(z))\right)^{\beta +n+2}\notag\\
    &=  \left(1- \alpha\vartheta(\overline{\alpha})+\xi(\overline{\kappa}\vartheta(\overline{\alpha})- \alpha)\right)^{\beta +n+2}\left(1-\vartheta(\xi)\overline{\kappa}\overline{\alpha}-z( \alpha -\overline{\kappa}\vartheta(\xi))\right)^{\beta +n+2}
  \end{align}
  Differentiate \eqref{e8} both sides w.r.t. $z$, we have
  \begin{align}\label{e9}
    &(\overline{\kappa}\vartheta(\overline{\alpha})- \alpha)\left(1- \alpha\vartheta(\overline{\alpha})+z(\overline{\kappa}\vartheta(\overline{\alpha})- \alpha)\right)^{\beta +n+1}\left(1-\vartheta(z)\overline{\kappa}\overline{\alpha}-\xi( \alpha -\overline{\kappa}\vartheta(z))\right)^{\beta +n+2}\notag\\
    &+ \left(1- \alpha\vartheta(\overline{\alpha})+z(\overline{\kappa}\vartheta(\overline{\alpha})- \alpha)\right)^{\beta +n+2}\overline{\kappa}(\overline{\alpha}-\xi)\vartheta'(z)\left(1-\vartheta(z)\overline{\kappa}\overline{\alpha}-\xi( \alpha -\overline{\kappa}\vartheta(z))\right)^{\beta +n+1}\notag\\
    &=\left(1- \alpha\vartheta(\overline{\alpha})+\xi(\overline{\kappa}\vartheta(\overline{\alpha})- \alpha)\right)^{\beta +n+2}\left(- \alpha+\overline{\kappa}\vartheta(\xi)\right)\left(1-\vartheta(\xi)\overline{\kappa}\overline{\alpha}-z( \alpha -\overline{\kappa}\vartheta(\xi))\right)^{\beta +n+1}
  \end{align}
  Implies
  \begin{align}\label{e9}
   &(\overline{\kappa}\vartheta(\overline{\alpha})- \alpha) \left(1-\vartheta(z)\overline{\kappa}\overline{\alpha}-\xi( \alpha -\overline{\kappa}\vartheta(z))\right)+ \left(1- \alpha\vartheta(\overline{\alpha})+z(\overline{\kappa}\vartheta(\overline{\alpha})- \alpha)\right)\overline{\kappa}(\overline{\alpha}-\xi)\vartheta'(z) \notag\\
   &= \left(- \alpha+\overline{\kappa}\vartheta(\xi)\right)\left(1-\vartheta(\xi)\overline{\kappa}\overline{\alpha}-z( \alpha -\overline{\kappa}\vartheta(\xi))\right)
  \end{align}

  Now substitute $z=\overline{\alpha}$ with $\kappa  \alpha =\overline{\alpha}$ in \eqref{e9}, we obtain

  \begin{equation}
      \vartheta(\xi)= \vartheta(\overline{\alpha})+ \dfrac{(1-|\alpha|^{2})(\xi-\overline{\alpha})\vartheta'(\overline{\alpha})}{\left(1-\vartheta(\overline{\alpha})\alpha+\xi(\overline{\kappa}\vartheta(\overline{\alpha})- \alpha)\right)}
  \end{equation}

  Conversely, assume that 
  \begin{equation}
      u(z)= \dfrac{(1-|\alpha|^{2})^{\beta+n+2} u^{(n)}(\overline{\alpha})(z-\overline{\alpha})^{n}}{\left(1- \alpha\vartheta(\overline{\alpha})+z(\overline{\kappa}\vartheta(\overline{\alpha})- \alpha)\right)^{\beta +n+2}}
  \end{equation}
  and \begin{equation}
     \vartheta(z)= \vartheta(\overline{\alpha})+ \dfrac{(1-|\alpha|^{2})(z-\overline{\alpha})\vartheta'(\overline{\alpha})}{\left(1-\vartheta(\overline{\alpha})\alpha+z(\overline{\kappa}\vartheta(\overline{\alpha})- \alpha)\right)}
  \end{equation}
  By our assumption $W^{n}_{u,\vartheta}$ is bounded on $\mathcal{A}^{2}_{\beta}.$ Hence, it is enough to verify that \eqref{e1} holds for any $z,\xi \in \mathbb{D}.$
  Therefore, we have
  \begin{align}
      &JW_{u_ {\alpha}, \vartheta_ {\alpha}}(W^{n}_{u,\vartheta})^{*}K_{\xi}(z)= \dfrac{(1-|\alpha|^{2})^{\frac{\beta+2}{2}}u(\xi)(\overline{\alpha}-z)^{n}\displaystyle\prod_{i=2}^{n+1}(\beta+i)}{\left(1-\vartheta(\xi)\overline{\kappa}\overline{\alpha}-z( \alpha -\overline{\kappa}\vartheta(\xi)\right)^{\beta +n+2}} \notag\\
      &= \dfrac{(1-|\alpha|^{2})^{\frac{3(\beta+2)}{2}+2n}\overline{\kappa}(\xi-\overline{\alpha})u^{(1)}(\overline{\alpha})(\overline{\alpha}-z)^{n}\displaystyle\prod_{i=2}^{n+1}(\beta+i)}{\Big\{\left(1-  \alpha  z-\overline{\kappa}(\overline{\alpha}-z)\vartheta(\overline{\alpha})\right)\left(1- \alpha\vartheta(\overline{\alpha})+(\overline{\kappa}\vartheta(\overline{\alpha})- \alpha)\xi\right)-\overline{\kappa}(\overline{\alpha})-z)(\left(1-|\alpha|^{2}\vartheta^{'}(\overline{\alpha})(\xi-\overline {\alpha})\right)\Big\}^{\beta +n+2} }
  \end{align}
  and 
  \begin{align}
    &W^{n}_{u,\vartheta}JW_{u_ {\alpha}, \vartheta_ {\alpha}}K_{\xi}(z)= \dfrac{(1-|\alpha|^{2})^{\frac{\beta+2}{2}}u(z)(\overline{\alpha}-\xi)^{n}\displaystyle\prod_{i=2}^{n+1}(\beta+i)}{\left(1-\vartheta(z)\overline{\kappa}\overline{\alpha}-\xi( \alpha -\overline{\kappa}\vartheta(z))\right)^{\beta +n+2}} \notag\\
    &=\dfrac{(1-|\alpha|^{2})^{\frac{3(\beta+2)}{2}+2n}\overline{\kappa}(\overline{\alpha}-z)u^{(1)}(\overline{\alpha})(\xi-\overline{\alpha})^{n}\displaystyle\prod_{i=2}^{n+1}(\beta+i)}{\Big\{\left(1- \alpha\xi-\overline{\kappa}(\overline{\alpha}-\xi)\vartheta(\overline{\alpha})\right)\left(1- \alpha\vartheta(\overline{\alpha})+(\overline{\kappa}\vartheta(\overline{\alpha})- \alpha)z\right)-\overline{\kappa}(\overline{\alpha})-\xi)\left(1-|\alpha|^{2}\vartheta^{'}(\overline{\alpha})(z-\overline {\alpha})\right)\Big\}^{\beta +n+2}} 
  \end{align}
  Since the numerator are same, it is enough to verify that denominator is also same.
This completes the proof.
  \end{proof}

  \begin{theorem}\label{2t}
Let $\vartheta$ be a member of the class $\mathfrak{S}(\mathbb{D})$, and let $u$ be a non-zero analytic function in $\mathcal{H}(\mathbb{D})$. Suppose that the operator $W^{n}_{u,\vartheta}$ is bounded on the weighted Bergman space $\mathcal{A}^{2}_{\beta}$. Then $W^{n}_{u,\vartheta}$ is self-adjoint if and only if the functions $u$ and $\vartheta$ take the form
\begin{equation*}
    u(z) = \frac{c_0 z^n}{(1 - c_1 z)^{\beta + n + 2}}, \quad \text{and} \quad \vartheta(z) = c_1 + \frac{c_2}{1 - c_1 z},
\end{equation*}
where $c_0, c_2 \in \mathbb{R}$, $c_0 \ne 0$, and $c_1 \in \mathbb{C}$.
\end{theorem}
\begin{proof}
First suppose that $W^{n}_{u, \vartheta}$ is self adjoint, we have
\begin{equation}\label{ee1}
  W^{n}_{u, \vartheta}K_{\xi}(z)= (W^{n}_{u, \vartheta} )^{*}K_{\xi}(z)
\end{equation}
we can calculate
\begin{equation}\label{ee2}(W^{n}_{u, \vartheta} )^{*}K_{\xi}= \overline{u(\xi)}\displaystyle\prod_{i=2}^{n+1}(\beta+i)\dfrac{z^{n}}{(1-\overline{\vartheta(\xi)}z)^{\beta+n+2}}
\end{equation}
and 
\begin{equation}\label{ee3}
   W^{n}_{u, \vartheta}K_{\xi}(z)= u(z)\displaystyle\prod_{i=2}^{n+1}(\beta+i)\dfrac{\xi^{n}}{(1-\overline{\xi}\vartheta(z))^{\beta+n+2}} 
\end{equation}
Therefore, we have
\begin{equation}\label{ee4}
    \dfrac{\overline{u(\xi)}z^{n}}{(1-\overline{\vartheta(\xi)}z)^{\beta+n+2}}= \dfrac{u(z)(\overline{\xi})^{n}}{(1-\overline{\xi}\vartheta(z))^{\beta+n+2}}
\end{equation}
Set $z=0$ gives $u(0)=0$, i.e., $u(z)= z^{j}g(z)$ for some $j\in \mathbb{N}$ and $g$ is analytic function with $g(0)\neq 0$.
Applying similar argument as in the proof of previous theorem, we get $j=n$ and $u(z)= z^{n}g(z)$, then 
\begin{equation}\label{ee5}
    \dfrac{\overline{g(\xi)}}{(1-\overline{\vartheta(\xi)}z)^{\beta+n+2}}=\dfrac{g(z)}{(1-\overline{\xi}\vartheta(z))^{\beta+n+2}}
\end{equation}
Set $\xi=0$ in \eqref{ee5}, we get
\begin{equation}\label{ee6}
    g(z)= \dfrac{\overline{g(0)}}{(1-\overline{\vartheta(0)}z)^{\beta+n+2}}
\end{equation}
If we substitute $z=0$ in \eqref{ee6}, we get 
$$ \overline{g(0)}= g(0)\;\;\; i.e.,\;\; g(0)\in \mathbb{R}.$$
Thus $u(z)= \dfrac{g(0) z^{n}}{(1-\overline{\vartheta(0)}z)^{\beta+n+2}}$.
From \eqref{ee4}, we have
\begin{equation}\label{ee7}
  (1-\vartheta(0)\overline{\xi})^{\beta+n+2}(1-\overline{\vartheta(\xi)}z)^{\beta+n+2} = (1-\overline{\vartheta(0)}z)^{\beta+n+2}  (1-\vartheta(z)\overline{\xi})^{\beta+n+2}
\end{equation}
Differentiate with respect to $z$, we get
\begin{align}\label{ee8}
    -&\overline{\vartheta(\xi)}(1-\vartheta(0)\overline{\xi})^{\beta+n+2}(1-\overline{\vartheta(\xi)}z)^{\beta+n++1}\notag\\
    &= -\overline{\vartheta(0)}(1-\overline{\vartheta(0)}z)^{\beta+n+1}  (1-\vartheta(z)\overline{\xi})^{\beta+n+2}-\overline{\xi}\vartheta'(z)(1-\overline{\vartheta(0)}z)^{\beta+n+2}  (1-\vartheta(z)\overline{\xi})^{\beta+n+1}
\end{align}
Set $z=0$, we have
\begin{align}\label{ee9}
    \overline{\vartheta(\xi)}(1-\vartheta(0)\overline{\xi})^{\beta+n+2}&= \overline{\vartheta(0)} (1-\vartheta(0)\overline{\xi})^{\beta+n+2}-\overline{\xi}\vartheta'(0)(1-\vartheta(0)\overline{\xi})^{\beta+n+1}\notag\\
    \vartheta(\xi)=\vartheta(0)+ \dfrac{\xi\overline{\vartheta'(0)}}{1-\overline{\vartheta(0)}\xi}
\end{align}
Then 
\begin{equation}\label{ee10}
  \vartheta'(\xi)= \dfrac{\overline{\vartheta'(0)}}{1-\overline{\vartheta(0)}\xi} 
\end{equation}
setting $\xi=0$, we get $\vartheta'(0)=\overline{\vartheta'(0)}$ implies $c_{0}=\vartheta'(0) \in \mathbb{R}$
Hence, if $W^{n}_{u, \vartheta}$ is self adjoint, then $u$ and $\vartheta$ are of the form
\begin{equation*}
    u(z)=\dfrac{c_{0}z^{n}}{(1-c_{1}z)^{\beta+n+2}}
\end{equation*}
and 
\begin{equation*}
    \vartheta(z)=c_{1}+\dfrac{c_{2}}{(1-c_{1}z)}
\end{equation*}
where $c_{0},c_{2} \in \mathbb{R}$ and $c_{1} \in \mathbb{C}$ with $c_{0}\neq 0$.\\
Conversely, assume that 
\begin{equation*}
    u(z)=\dfrac{c_{0}z^{n}}{(1-c_{1}z)^{\beta+n+2}}
\end{equation*}
and 
\begin{equation*}
    \vartheta(z)=c_{1}+\dfrac{c_{2}}{(1-c_{1}z)}
\end{equation*}
where $c_{0},c_{2} \in \mathbb{R}$ and $c_{1} \in \mathbb{C}$ with $c_{0}\neq 0$.\\
One can easily calculate that
\begin{equation*}
    (W^{n}_{u, \vartheta})^{*}K_{\xi}(z)= \dfrac{\displaystyle\prod_{i=2}^{n+1}(\beta+i) c_{0}\overline{\xi}^{n}z^{n}}{(1-c_{1}\overline{\xi}-\overline{c_{1}}z(1-c_{1}\overline{\xi})-c_{2}\overline{\xi}z)^{\beta+n+2}}= W^{n}_{u, \vartheta}K_{\xi}(z)
\end{equation*}
Thus, $W^{n}_{u, \vartheta}$ is self-adjoint.
This completes the proof.
\end{proof}

\section{Normality and spectral properties generalized product type operators}
In this section we investigate the Normality and spectrum of the normal weighted composition-differentiation operators.
\begin{theorem}\label{3}
Let $\vartheta \in \mathfrak{S}(\mathbb{D})$ and let $u$ be a non-zero function in $\mathcal{H}(\mathbb{D})$ such that the operator $W^{n}_{u,\vartheta}$ is bounded on the space $\mathcal{A}^{2}_{\beta}$. Suppose $W^{n}_{u,\vartheta}$ is complex symmetric with respect to the conjugation $J W_{u_\alpha, \vartheta_\alpha}$, and that $\vartheta(\overline{\alpha}) = \overline{\alpha}$. Then, $W^{n}_{u,\vartheta}$ is normal if and only if the functions $u$ and $\vartheta$ are of the form
\[
u(z) =  \alpha z^n \quad \text{and} \quad \vartheta(z) = c z,
\]
where $ \alpha = u^{(n)}(0)$ and $c = \vartheta'(0)$.
\end{theorem}

\begin{proof}
 Since $W^{n}_{u,\vartheta}$ is complex symmetric with respect to conjugation $JW_{u_ {\alpha},\vartheta_ {\alpha}}$ and $\vartheta(\overline{\alpha})=\overline{\alpha}$, then by Theorem\ref{1t}, we have
 \begin{equation}
     \vartheta(z)= \overline{\alpha}+\vartheta'(\overline{\alpha})(z-\overline{\alpha}) \;\; \text{and}\;\; u(z)=u^{(n)}(\overline{\alpha})(z-\overline{\alpha})^{n}
 \end{equation}
 \begin{align}
     \|(W^{n}_{u,\vartheta})^{*}\vartheta_{k}\|^{2}&= \sum_{j=0}^{n}\Big|{\langle}(W^{n}_{u,\vartheta})^{*}\vartheta_{k}, \vartheta_{j}\rangle\Big|^{2}\notag\\
     &= \sum_{j=0}^{n}\Big|{\langle}\vartheta_{k}, (W^{n}_{u,\vartheta})\vartheta_{j}\rangle\Big|^{2}\notag\\
     &=\sum_{j=0}^{n}\Big|{\langle}\vartheta_{k}, u.\vartheta_{j}^{(n)}(\vartheta)\rangle\Big|^{2}\notag\\
     &=\sum_{j=0}^{n}\Bigg|{\Bigg\langle}\sqrt{\dfrac{\Gamma(k+2+\beta)}{k! \Gamma(2+\beta)}}z^{k}, \sqrt{\dfrac{\Gamma(j+2+\beta)}{j! \Gamma(2+\beta)}} \dfrac{j!}{(j-n)!} u^{(1)}(\overline{\alpha})(z-\overline{\alpha})^{n}\left(\overline{\alpha}+(z-\overline{\alpha})\vartheta'(\overline{\alpha})\right)^{j-n}\Bigg\rangle\Bigg|^{2} 
 \end{align}
and 
\begin{align}
     \|(W^{n}_{u,\vartheta})\vartheta_{k}\|^{2}&= \sum_{j=0}^{n}\Big|{\big\langle}(W^{n}_{u,\vartheta})\vartheta_{k}, \vartheta_{j}\big\rangle\Big|^{2}\notag\\
     &=\sum_{j=0}^{n}\Big|{\Big\langle}u.\vartheta_{k}^{(n)}(\vartheta),\vartheta_{j}\Big\rangle\Big|^{2}\notag\\
     &=\sum_{j=0}^{n}\Bigg|{\Bigg\langle}
     \sqrt{\dfrac{\Gamma(k+2+\beta)}{k! \Gamma(2+\beta)}}\dfrac{k!}{(k-n)!}u^{(1)}(\overline{\alpha})(z-\overline{\alpha})^{n}\left(\overline{\alpha}+(z-\overline{\alpha})\vartheta'(\overline{\alpha})\right)^{k-n},  \sqrt{\dfrac{\Gamma(j+2+\beta)}{j! \Gamma(2+\beta)}}z^{j}\Bigg\rangle\Bigg|^{2} 
 \end{align}

Above calculation implies that 
\begin{equation}
  \|(W^{n}_{u,\vartheta})^{*}\vartheta_{k}\|^{2}= \|(W^{n}_{u,\vartheta})\vartheta_{k}\|^{2}= \sum_{j=0}^{n}\Big|u^{(1)}(0)(\vartheta'(0))^{k-n}\Big|^{2}\left(\dfrac{k!}{(k-n)!}\right)^{2}  
\end{equation}
if and only if $u(z)=  \alpha z^{n}$ and $\vartheta(z)= c$, where $ \alpha = u^{(1)}(0)$ and $c=\vartheta'(0).$
  
\end{proof}
\begin{theorem}\label{4}
Consider $\vartheta \in \mathfrak{S}(\mathbb{D})$ and a non-zero function $u \in \mathcal{H}(\mathbb{D})$ such that the operator $W^{n}_{u,\vartheta}$ is bounded on $\mathcal{A}^{2}_{\beta}$. Assume that $W^{n}_{u,\vartheta}$ is complex symmetric with respect to the conjugation $J W_{u_\alpha, \vartheta_\alpha}$, and that the derivative $\vartheta'( \overline{\alpha} )$ vanishes. Then $W^{n}_{u,\vartheta}$ is normal if and only if the functions $u$ and $\vartheta$ are given by
$$
u(z) =  \alpha z^{n} \quad \text{and} \quad \vartheta(z) = c z,
$$
where $ \alpha = u'(0)$ and $c = \vartheta'(0)$.
\end{theorem}

\begin{proof}
 Since $W^{n}_{u,\vartheta}$ is complex symmetric with respect to conjugation $JW_{u_ {\alpha},\vartheta_ {\alpha}}$ and $\vartheta(\overline{\alpha})=\overline{\alpha}$, then by Theorem\ref{1t}, we have
 \begin{equation}
     \vartheta(z)= \vartheta(\overline{\alpha}) \;\; \text{and}\;\; u(z)=\dfrac{(1- |\alpha|^{2})^{\beta+n+2} u^{(n)}(\overline{\alpha})(z-\overline{\alpha})^{n}}{\left(1- \alpha\vartheta(\overline{\alpha})+z(\overline{\kappa}\vartheta(\overline{\alpha})- \alpha)\right)^{\beta +n+2}}
 \end{equation}
 \begin{align}
     \|(W^{n}_{u,\vartheta})^{*}\vartheta_{k}\|^{2}&= \sum_{j=0}^{n}\Big|{\langle}(W^{n}_{u,\vartheta})^{*}\vartheta_{k}, \vartheta_{j}\rangle\Big|^{2}\notag\\
     &= \sum_{j=0}^{n}\Big|{\langle}\vartheta_{k}, (W^{n}_{u,\vartheta})\vartheta_{j}\rangle\Big|^{2}\notag\\
     &=\sum_{j=0}^{n}\Big|{\langle}\vartheta_{k}, u.\vartheta_{j}^{(n)}(\vartheta)\rangle\Big|^{2}\notag\\
     &=\sum_{j=0}^{n}\Bigg|{\Bigg\langle}\sqrt{\dfrac{\Gamma(k+2+\beta}{k! \Gamma(2+\beta)}}z^{k},  
     \dfrac{j!u^{(1)}(\overline{\alpha})(z-\overline{\alpha})^{n}}{(j-n)!(\overline{\alpha}+(z-\overline{\alpha})\vartheta'(\overline{\alpha}))}\Bigg\rangle\Bigg|^{2} 
 \end{align}
and 
\begin{align}
     \|(W^{n}_{u,\vartheta})\vartheta_{k}\|^{2}&= \sum_{j=0}^{n}\Big|{\langle}(W^{n}_{u,\vartheta})\vartheta_{k}, \vartheta_{j}\rangle\Big|^{2}\notag\\
     &=\sum_{j=0}^{n}\Big|{\Big\langle}u.\vartheta_{k}^{(n)}(\vartheta),\vartheta_{j}\Big\rangle\Big|^{2}\notag\\
     &=\sum_{j=0}^{n}\Bigg|{\Bigg\langle}
     u^{(1)}(\overline{\alpha})(z-\overline{\alpha})^{n}\dfrac{k!}{(k-n)!(\overline{\alpha}+(z-\overline{\alpha})\vartheta'(\overline{\alpha})},  \sqrt{\dfrac{\Gamma(j+2+\beta)}{j! \Gamma(2+\beta)}}z^{j}\Bigg\rangle\Bigg|^{2} 
 \end{align}

Above calculation implies that 
\begin{equation}
  \|(W^{n}_{u,\vartheta})^{*}\vartheta_{k}\|^{2}= \|(W^{n}_{u,\vartheta})\vartheta_{k}\|^{2}= \sum_{j=0}^{n}\Big|u^{(1)}(0)(\vartheta'(0))^{k-n}\Big|^{2}\left(\dfrac{k!}{(k-n)!}\right)^{2}  
\end{equation}
if and only if $u(z)=  \alpha z^{n}$ and $\vartheta(z)= cz$, where $ \alpha = u^{(1)}(0)$ and $c=\vartheta'(0).$
  
\end{proof}

\begin{theorem}\label{5}
Let $u(z) =   \alpha z^{n}$ and $\vartheta(z) = c z$, where  $\alpha, c \in \mathbb{C}$ and $n \in \mathbb{N}$. Suppose that the operator $W^{n}_{u,\vartheta}$ is bounded on $\mathcal{A}^2_\beta$. Then the spectrum of $W^{n}_{u,\vartheta}$ consists of the eigenvalues
$$
\{0\} \cup \left\{ \frac{m!}{(m - n)!} \, \alpha\, c^{m - n} : m \geq n \right\}.
$$
Moreover, for each $m \geq n$, the monomial $z^m$ is an eigenvector associated with the eigenvalue $\dfrac{m!}{(m-n)!} \alpha c^{m-n}$.
\end{theorem}

\begin{proof}
Assume that $u(z) =   \alpha z^{n}$ and $\vartheta(z) = c z$. Then applying the operator $W^n_{u, \vartheta}$ to the monomial $z^m$ yields
\begin{equation}
W^n_{u,\vartheta} z^m =   \alpha z^{n} \cdot \frac{d^n}{dz^n} \left( z^m \circ \vartheta \right) =   \alpha z^{n} \cdot \frac{d^n}{dz^n} (c z)^m = \frac{m!}{(m-n)!} \, \alpha\, c^{m - n} z^m,
\end{equation}
for all integers $m \geq n$. Hence, each $z^m$ with $m \geq n$ is an eigenvector corresponding to the eigenvalue $\dfrac{m!}{(m - n)!}  \alpha c^{m - n}$, and $z^m$ is mapped to a scalar multiple of itself.

To identify all eigenvalues, let $\kappa$ be an eigenvalue of $W^n_{u,\vartheta}$ with a non-zero eigenfunction $f \in \mathcal{A}^{2}_{\beta}$ satisfying
\begin{equation}\label{q1}
\kappa f(z) =  \alpha c^n z^n f^{(n)}(c z).
\end{equation}  
If $f(0) \neq 0$, then plugging $z = 0$ into \eqref{q1} gives $\kappa f(0) = 0$, so $\kappa = 0$.

If $f(0) = 0$ but $f'(0) \neq 0$, differentiate both sides of \eqref{q1}, we get
$$
\kappa f'(z) = n  \alpha c^n z^{n-1} f^{(n)}(c z) +  \alpha c^{n+1} z^n f^{(n+1)}(c z).
$$
Evaluating at $z = 0$ gives $\kappa f'(0) = 0 \Rightarrow \kappa = 0$ again.

Now, $n-$times differentiation of both sides of \eqref{q1} with respect to $z$ gives
\begin{equation}
  \kappa f^{n}(z)=  \alpha c^{n} \Bigg[\displaystyle\sum_{i=0}^{n}c^{i} \binom{n}{i}\frac{n!}{i!} f^{(n+i)}(cz) \Bigg] 
\end{equation}
If $f^{n}(0)\neq 0,$ then $\kappa=n! \alpha c^{n}$. If $f^{n}(0) = 0,$ then $n+1-$times differentiation of both sides of \eqref{q1} with respect to $z$ gives
\begin{equation}
  \kappa f^{n+1}(z)=  \alpha c^{n} \Bigg[\displaystyle\sum_{i=0}^{n+1}c^{i} \binom{n+1}{i}\frac{(n+1)!}{(n+1-i)!} \frac{n!}{i!}f^{(n+i)}(cz) \Bigg] 
\end{equation}

If $f^{n+1}(0)\neq 0,$ then $\kappa=(n+1)! \alpha c^{n+1}$. If $f^{n+1}(0) = 0,$ then $n+2-$times differentiation of both sides of \eqref{q1} with respect to $z$ gives
\begin{equation}
  \kappa f^{n+2}(z)=  \alpha c^{n} \Bigg[\displaystyle\sum_{i=0}^{n+2}c^{i} \binom{n+2}{i}\frac{(n+2)!}{(n+2-i)!} \frac{n!}{i!}f^{(n+i)}(cz) \Bigg] 
\end{equation}
   If $f^{n+1}(0)\neq 0,$ then $\kappa=\frac{(n+2)!}{2!} \alpha c^{n+2}$. If $f^{n+2}(0) = 0,$ then we proceed the same as above and we get the eigenvalues of $W^{n}_{u,\vartheta}$ given by $$\{0\} \bigcup \Bigg\{\dfrac{m!}{(m-n)!} \alpha c^{m-n};\;\;m\ge n \Bigg\}.$$ Therefore, $z^{m}$ is an eigen vector of $W^{n}_{u,\vartheta}$ with respect to the eigenvalue $$\Bigg\{\dfrac{m!}{(m-n)!} \alpha c^{m-n};\;\; m\ge n \Bigg\}.$$
   
\end{proof}

\section{conflict of interest}
There is no conflict of interest.

\end{document}